\documentclass{amsart}

\usepackage{amsmath, amsfonts, amssymb} 
\usepackage{enumerate}
\usepackage{enumitem}
\usepackage{hyperref}
\usepackage{graphicx}
\usepackage[all]{xy}
\usepackage{wrapfig}
\usepackage{fancyvrb}
\usepackage[mathscr]{euscript}
\usepackage[T1]{fontenc}
\usepackage{listings}
\usepackage{tikz}
\usepackage{amsthm}
\usepackage{dsfont}
\usepackage{tikz-cd}
\usepackage{adjustbox}
\usepackage{centernot}
\usepackage{mathtools}
\usepackage{cite}
\usepackage{bbm}
\usepackage{stmaryrd}

\usepackage[margin=3cm]{geometry}

\let\oldtocsection=\tocsection

\let\oldtocsubsection=\tocsubsection

\let\oldtocsubsubsection=\tocsubsubsection

\renewcommand{\tocsection}[2]{\hspace{0em}\oldtocsection{#1}{#2}}
\renewcommand{\tocsubsection}[2]{\hspace{1em}\oldtocsubsection{#1}{#2}}
\renewcommand{\tocsubsubsection}[2]{\hspace{2em}\oldtocsubsubsection{#1}{#2}}

\title{Relative Serre Duality for Coxeter Groups}
\author{Colton Sandvik}
\date{}

\newcommand{\C}{\mathbb{C}}

\newcommand{\Z}{\mathbb{Z}}

\renewcommand{\k}{\mathbbm{k}}

\newcommand{\scrF}{\mathcal{F}}
\newcommand{\scrG}{\mathcal{G}}
\newcommand{\scrH}{\mathcal{H}}
\newcommand{\scrC}{\mathcal{C}}

\newcommand{\scrE}{\mathcal{E}}

\newcommand{\scrA}{\mathcal{A}}

\newcommand{\uw}{\underline{w}}

\DeclareMathOperator{\Hom}{Hom}

\newcommand{\grbim}[1]{#1\textnormal{-bim}^{\Z}}

\DeclareMathOperator{\Free}{Free}
\newcommand{\BE}{D^m}

\newcommand{\fr}[1]{\mathfrak{#1}}
\DeclareMathOperator{\op}{op}
\DeclareMathOperator{\fg}{fg}
\DeclareMathOperator{\id}{id}
\DeclareMathOperator{\idem}{idem}

\renewcommand{\emptyset}{\varnothing}

\newcommand{\wistar}{\stackrel{I}{\star}}
\newcommand{\DD}{\mathbb{D}}
\DeclareMathOperator{\supp}{supp}
\DeclareMathOperator{\Sym}{Sym}

\newtheorem{corollary}{Corollary}[section]
\newtheorem{lemma}{Lemma}[section]
\newtheorem{proposition}{Proposition}[section]
\newtheorem{conjecture}{Conjecture}[section]
\newtheorem{theorem}{Theorem}[section]

\theoremstyle{definition}

\newtheorem{remark}{Remark}[section]

\newenvironment{midsecproof}[1]{\vspace{\topsep} \noindent \textit{Proof of #1.}}{\hfill$\square$}

\tikzstyle{row}=[-, draw=black, thick]
\tikzstyle{dashed_row}=[-, draw=black, dashed]

\DeclareMathOperator{\FT}{FT}

\begin{document}

        \begin{abstract}
                It was conjectured by Gorsky, Hogancamp, Mellit, and Nakagane that the left and right adjoints of the parabolic induction functor between homotopy categories of Soergel bimodules associated to a finite Coxeter group are related by the relative full twist.
                Several cases of this conjecture are known including for symmetric groups, crystallographic Coxeter groups, and dihedral groups.
                We prove this conjecture in complete generality using the theory of Abe--Bott--Samelson bimodules and the Achar--Riche--Vay mixed derived category. 
        \end{abstract}

        \maketitle

        \section{Introduction}

        \subsection{Bott--Samelson Bimodules}

        Let $(W,S)$ be a Coxeter system and $\fr{h}$ be a realization of $W$ over a noetherian domain $\k$ of finite global dimension.
        Associated to the pair $(\fr{h}, W)$ is the category of Abe--Bott--Samelson bimodules $\scrA^{\oplus} (\fr{h}, W)$ introduced by \cite{Abe19}.
        This is a (non-full) additive subcategory of graded $R = \Sym (\fr{h}^*)$-bimodules which recovers the usual Bott--Samelson bimodules when $\fr{h}$ is reflection faithful.
        Abe--Bott--Samelson bimodules can be used to define the \emph{mixed derived category},
        \[\BE (\fr{h}, W) \coloneq K^b \scrA^{\oplus} (\fr{h}, W).\]
        This category admits a monoidal structure $\star$ inherited from the tensor product of graded $R$-bimodules.
        Based on similar constructions using parity sheaves, Achar--Riche--Vay \cite{ARV} gave an in-depth study of the mixed derived category.
        In particular, they showed that many constructions and features encountered in the theory of mixed $\ell$-adic sheaves have purely combinatorial analogues for the mixed derived category.
        
        The mixed derived category was first studied in knot theory. For example, Rouquier gave an action of the braid group on the mixed derived category \cite{Rou}.
        When $W= S_n$, this action is the foundation for the construction of link invariants such as  Khovanov--Rozansky homology  \cite{Kho, KR}.
        Since their introduction, the mixed derived category has also seen considerable study in modular representation theory.
        For example, it is a crucial component in proving the tilting character formulas for reductive groups developed by Achar--Makisumi--Riche--Williamson \cite{AMRW}.

        \subsection{Relative Serre Duality}
        Let $I \subseteq S$ be a subset of simple reflections, and write $W_I$ for the parabolic subgroup generated by $I$. The realization $\fr{h}$ for $W$ can also be viewed as a realization for $W_I$. 
        There is a monoidal functor
        \[\iota : \BE (\fr{h}, W_I) \to \BE (\fr{h}, W)\]
        called \emph{parabolic induction} which is induced by the inclusion $W_I \subseteq W$. 

        It is shown in \cite[Corollary A.8]{GHMN}, that parabolic induction admits a left adjoint $\iota^L$ and a right adjoint $\iota^R$.
        Another perspective on these adjoints comes from the recollement formalism of \cite{ARV} which we heavily use in this paper.
        In \cite{GHMN}, it was conjectured that the left and right adjoints are related under $\star$ by the \emph{relative full twist}-- a certain object $\FT_{W, I} \in \BE (\fr{h}, W)$ defined using Rouquier's braid group action \cite{Rou}.
        More precisely, they conjectured the following.

        \begin{conjecture}[{\!\!\cite[Conjecture 1.8]{GHMN}}]\label{conj:rel_serre_duality_intro}
                Assume that $W$ is finite.
                \begin{enumerate}
                        \item There are natural isomorphisms of functors 
                \[
                        \iota^L (\FT_{W, I} \star -) \cong  \iota^R \cong \iota^L (- \star \FT_{W, I}).
                \]
                        \item $\FT_{W,I}$ naturally commutes with objects of $\BE (\fr{h}, W_I)$.
                \end{enumerate}
        \end{conjecture}

        We will briefly discuss some history on known cases of Conjecture \ref{conj:rel_serre_duality_intro}.
        \begin{enumerate}
                \item Let $W = S_n$ and $W_I = S_r \times (S_1)^{n-r}$ for some $r \leq n$. Consider the realization $\fr{h} = \k^{\oplus n}$ where $W$ acts on $\fr{h}$ by permuting the coordinates. In this case, Conjecture \ref{conj:rel_serre_duality_intro} was proved in \cite{GHMN}.
                \item Q. Ho and P. Li \cite{HLnew} made substantial progress on Conjecture \ref{conj:rel_serre_duality_intro} using their mixed sheaf formalism introduced in \cite{HLgeom}. 
                In particular, they proved in \cite{HLnew} that Conjecture \ref{conj:rel_serre_duality_intro} (1) holds when $W$ is the Weyl group of a reductive algebraic group and $\fr{h}$ is its associated Cartan realization over $\C$.
                \item When $W$ is a dihedral group, Conjecture \ref{conj:rel_serre_duality_intro} (1) was proved by C. Li in \cite{Li} under some minor non-degeneracy conditions.
                \item Conjecture \ref{conj:rel_serre_duality_intro} (2) is already known provided $W$ does not have a parabolic subgroup of type $H_3$.
                More precisely, Elias and Hogancamp \cite{EH} proved the much stronger statement that the full twist $\FT_W$ naturally commutes with objects of $\BE (\fr{h}, W)$. 
                The restriction on $W$ is a facet of working with the Elias--Williamson diagrammatic category where the 3-color relation is not known in type $H_3$.
        \end{enumerate}
        In this paper, we prove Conjecture \ref{conj:rel_serre_duality_intro} in its entirety without any assumptions on $W$ or $\fr{h}$ other than the minor assumptions on $\fr{h}$ needed to have a well-behaved theory of Abe--Bott--Samelson bimodules (cf., \cite{Abe21}).
        We refer to Theorem \ref{thm:rel_serre_duality} and Corollary \ref{cor:rel_full_twist_centrality} for precise statements. 

        Before continuing, it is worth noting that the author expects that Elias and Hogancamp's result on the centrality of the full twist \cite{EH} should also hold in the setting of this paper.
        It seems likely that this can be proved using similar methods to \emph{loc. cit.} using Abe--Bott--Samelson bimodules instead of the diagrammatic category.

        \begin{conjecture}\label{conj:ft_commutes}
                $\FT_W$ naturally commutes with objects of $\BE (\fr{h}, W)$ under $\star$.
        \end{conjecture}

        The overarching strategy for proving Conjecture \ref{conj:rel_serre_duality_intro} is heavily influenced by \cite{HLnew}.
        The mixed derived category formalism of \cite{ARV} will allow many of their arguments to be adapted even when $W$ is not crystallographic.
        We briefly mention where our proofs differ.
        
        \begin{enumerate}
                \item The argument for $\BE (\fr{h}, W)$ being a rigid monoidal category in \cite{HLnew} heavily depends on geometry which cannot be used for general Coxeter groups.
                Instead, we give an elementary proof of rigidity using the theory of Abe--Bott--Samelson bimodules. 
                As part of our proof, we also give some applications to the structure of the ``sheaf-functors'' in mixed derived categories. 
                \item The present paper is written solely using ordinary categories rather than $\infty$-categories. 
                This decision is largely necessitated by \cite{ARV} working with triangulated categories.
                The author expects that all the constructions in \emph{loc. cit.} can be adapted to use $\infty$-categories, but this is beyond the scope of this paper.
                The only place where $\infty$-categorical machinery is crucially used in \cite{HLnew} is to reduce proving their main result to the case of $I = \emptyset$.
                Instead, our argument is more closely based on an earlier pre-print \cite{HLold} which did not make such a reduction.
                Unfortunately, the argument in the original pre-print has an error which has necessitated new proofs in some places.
        \end{enumerate}

        \subsection{Acknowledgements}
        The author would like to thank Simon Riche for suggesting this project, numerous valuable comments, and carefully reading an earlier draft. The author also thanks Pramod Achar for helpful discussions.
        The author was partially supported by NSF Grant No. DMS-2231492.

        \section{Hecke Categories}

        For this section, we will allow $(W, S)$ to be an arbitrary (possibly infinite) Coxeter system.
        Let $(\fr{h}, \{\alpha_s\}_{s \in S}, \{\alpha_s^{\vee}\}_{s \in S})$ be a realization of $W$ over $\k$ satisfying \cite[Assumption 3.5]{Abe21}.
        Consider the graded algebra $R = \Sym (\fr{h}^*)$ where $\fr{h}^*$ is placed in degree $2$.
        We will also consider the localization $Q$ of the ring $R$ with respect to the multiplicative subset generated by $\{w (\alpha_s) : s\in S, w \in W\}$.
        This ring is $\Z$-graded where $\frac{1}{w (\alpha_s)}$ is in degree $-2$.

        Let $(X, \preceq)$ be a poset. We will say that a subset $Y \subseteq X$ is \emph{closed} if for all $x,x' \in X$ with $x' \in Y$ and $x \preceq x'$ we have $x \in Y$.
        These closed sets define a topology on $X$. In particular, we can make sense of open and locally closed subsets of $X$.
        We will frequently apply this terminology to the case of $W$ endowed with its Bruhat order. In this setting, if $Y \subseteq W$ is closed (resp. open), then its inverse $Y^{-1} \coloneq \{ y^{-1} : y \in Y \}$ is also closed (resp. open).  

        \subsection{Abe--Bott--Samelson Bimodules}

        We will review Abe's theory of Bott--Samelson bimodules introduced in \cite{Abe19}.

        Let $\scrC (\fr{h}, W)$ denote the category whose objects are triples
        \[(M, (M_Q^w)_{w \in W}, \xi_M)\]
        where $M$ is a graded $R$-bimodule, each $M_Q^w$ is a graded $(R,Q)$-bimodule such that $m \cdot f = w(f) \cdot m$ for any $m \in M_Q^w$ and $f \in R$, this bimodule being 0 for all but finitely many $w\in W$, and $\xi_M : M \otimes_R Q \stackrel{\sim}{\to} \oplus_{w \in W} M_Q^w$ is an isomorphism of graded $(R,Q)$-bimodules.
        A morphism in $\scrC( \fr{h}, W)$ from $(M, (M_Q^w)_{w \in W}, \xi_M)$ to $(N, (N_Q^w)_{w \in W}, \xi_N)$ consists of a graded $R$-bimodule map $\varphi : M \to N$ such that 
        \[(\xi_N \circ (\varphi \otimes_R Q) \circ \xi_M^{-1}) (M_Q^w) \subset N_Q^w\]
        for all $w \in W$.
        This category has a natural monoidal structure induced by $\otimes_R$ with unit $R$. It also admits a shift-of-grading functor $(1)$ defined by $(M(1))^n = M^{n+1}$ and $(M_Q^w (1))^n = (M_Q^w)^{n+1}$.
        For simplicity, we often write $M$ instead of the triple $(M, (M_Q^w)_{w \in W}, \xi_M)$.

        Given $M, N \in \scrC (\fr{h}, W)$, we will write
        \begin{equation}\label{eq:graded_Hom_space}
                \Hom_{\scrC (\fr{h}, W)}^{\bullet} (M, N) \coloneq \bigoplus_{n \in \Z} \Hom_{\scrC (\fr{h}, W)} (M,N (n))
        \end{equation}
        which we may view as a graded $R$-bimodule.   
        Note that in \cite{Abe19}, the Hom-spaces in $\scrC (\fr{h}, W)$ are defined by (\ref{eq:graded_Hom_space}). Some discussion on these two points-of-view can be found in \cite[\S2]{ARV}.

        Let $s \in S$, and write $R^s$ for the subalgebra of $s$-invariant elements in $R$. 
        By Demazure surjectivity, we can choose some $\delta_s \in \fr{h}^*$ such that $\langle \delta_s, \alpha_s^{\vee} \rangle = 1$.
        The graded $R$-bimodule $B_s \coloneq R \otimes_{R^s} R (1)$ can be upgraded to an object in $\scrC (\fr{h}, W)$. In particular, if $w \notin \{e,s\}$, we set $(B_s)_Q^w = 0$, and otherwise, 
        \[(B_s)_Q^e = Q (\delta_s \otimes 1 - 1 \otimes s(\delta_s)) \qquad\text{and}\qquad (B_s)_Q^s = Q (\delta_s \otimes 1 - 1 \otimes \delta_s).\]
        
        We define a category $\scrA^{\oplus} (\fr{h}, W)$ as the smallest full subcategory of $\scrC (\fr{h}, W)$ which contains the monoidal unit $R$, the objects $(B_s)_{s \in S}$, and is stable under the monoidal product $\otimes_R$, direct sums $\oplus$, and the shift functor $(1)$.
        Given a sequence $\uw = (s_1,\ldots, s_k)$ of simple reflections, we can define a graded $R$-bimodule
        \[B_{\uw} \coloneq B_{s_1} \otimes_R B_{s_2} \otimes_R \ldots \otimes_R B_{s_k},\]
        which is viewed as an object of $\scrA^{\oplus} (\fr{h}, W)$.

        Consider the contravariant functor
        \[\DD : \scrA^{\oplus} (\fr{h}, W)^{\op} \to \scrA^{\oplus} (\fr{h}, W)\]
        defined on objects $M \in \scrA^{\oplus} (\fr{h}, W)$ by
        \[\DD (M) = \Hom_{\textnormal{mod}^{\Z}\textnormal{-}R}^{\bullet} (M, R) \qquad\text{and}\qquad \DD(M)_Q^w = \Hom_{\textnormal{mod}^{\Z}\textnormal{-}Q}^{\bullet} (M_Q^w, Q).\]
        We call $\DD$ the \emph{Verdier dual functor}. Its properties are studied in \cite[\S 2.6]{Abe19}. 
        For a sequence $\uw = (s_1, \ldots, s_k)$, there is an isomorphism $\DD (B_{\uw} (n)) \cong B_{\uw} (-n)$. Moreover, $\DD$ is an involution of $\scrA^{\oplus} (\fr{h}, W)$ in the sense that $\DD \circ \DD \cong \id$.

        \subsection{Mixed Derived Categories}

        We will recall the ``mixed derived category'' formalism from \cite{ARV}. Our approach differs from \emph{loc. cit.} in that we use Abe--Bott--Samelson bimodules instead of Elias--Williamson diagrammatics to allow for type $H_3$ parabolic subgroups of $W$.
        This does not have any measurable effect on the theory. In particular, the key feature used in the Achar--Riche--Vay formalism is an object-adapted cellular structure on $\scrA^{\oplus} (\fr{h}, W)$.
        One can prove that the double leaves basis constructed in \cite{Abe19} endows $\scrA^{\oplus} (\fr{h}, W)$ with such a structure. 
        For more details, see \cite[\S4.7]{Sandvik} and \cite[Appendix B]{Sandvik}.

        Define the \emph{mixed derived category}
        \[\BE (\fr{h}, W) \coloneq K^b \scrA^{\oplus} (\fr{h}, W).\]
        The mixed derived category has a monoidal structure inherited from $\scrA^{\oplus} (\fr{h}, W)$. We denote the monoidal product by $\star$.

        \begin{remark}
                When $\k$ is a complete local ring, there is a natural equivalence of categories
                \[\BE (\fr{h}, W) \cong K^b \scrA_{\idem}^{\oplus} (\fr{h}, W),\]
                where $\scrA_{\idem}^{\oplus} (\fr{h}, W)$ denotes the category of Abe--Soergel bimodules, i.e., the category obtained from taking the idempotent completion of $\scrA^{\oplus} (\fr{h}, W)$.
                If $\fr{h}$ is additionally reflection faithful, then this recovers the bounded homotopy category of Soergel bimodules appearing in the original conjecture of \cite{GHMN}. 
        \end{remark}

        Let $X \subseteq W$ be a closed subset. We define a full subcategory
        \[\scrA^{\oplus}_{ X} (\fr{h}, W) \subseteq \scrA^{\oplus} (\fr{h}, W)\]
        generated under direct sums by objects of the form $B_{\uw} (n)$ for $n \in \Z$ and $\uw$ a reduced expression for an element in $X$.
        More generally, if $X \subseteq W$ is locally closed, we can write $X = X_0 \setminus X_1$ where $X_1 \subseteq X_0 \subseteq W$ are closed.
        We set
        \[\scrA^{\oplus}_{ X_0, X_1} (\fr{h}, W) \coloneq \scrA^{\oplus}_{ X_0} (\fr{h}, W) /\!\!/ \scrA^{\oplus}_{ X_1} (\fr{h}, W),\]
        where the right-hand side is the ``naive quotient'' defined as follows. For $M, N \in \scrA^{\oplus}_{ X_0} (\fr{h}, W)$, write $\mathfrak{F}_{X_1} (M,N) \subset \Hom_{\scrA^{\oplus} (\fr{h},W)}^{\bullet} (M,N)$ for the submodule of morphisms which factor through $\scrA^{\oplus}_{ X_1} (\fr{h}, W)$.
        The objects in $\scrA^{\oplus}_{ X_0, X_1} (\fr{h}, W)$ are the same as $\scrA^{\oplus}_{ X_0} (\fr{h}, W)$. The morphism spaces are defined by
        \[ \Hom_{\scrA_{X_0, X_1}^{\oplus} (\fr{h},W)} (M,N) \coloneq \left(\Hom_{\scrA^{\oplus} (\fr{h},W)}^{\bullet} (M,N) / \mathfrak{F}_{X_1} (M,N) \right)^0. \]
        By \cite[Lemma 4.3]{ARV}, $\scrA^{\oplus}_{ X_0, X_1} (\fr{h}, W)$ only depends on $X_0 \setminus X_1$ up to a canonical equivalence of categories. 
        As a result, we will write $\scrA^{\oplus}_{ X} (\fr{h}, W) = \scrA^{\oplus}_{ X_0, X_1} (\fr{h}, W)$.
        
        The category $\scrA_X^{\oplus} (\fr{h}, W)$ retains an internal shift-of-grading functor
        \[(1) : \scrA_X^{\oplus} (\fr{h}, W) \to \scrA_X^{\oplus} (\fr{h}, W).\]
        Likewise, since $\DD (B_{\uw}) = B_{\uw}$ for any expression $\uw$, the Verdier dual functor induces an involution 
        \[\DD : \scrA^{\oplus}_{ X} (\fr{h}, W)^{\op} \to  \scrA^{\oplus}_{ X} (\fr{h}, W).\] 

        \subsection{Pushforward and Pullback Functors}

        For each locally closed subset $X \subseteq W$, one can associate a triangulated category
        \[\BE_X (\fr{h}, W) \coloneq K^b \scrA^{\oplus}_{ X} (\fr{h}, W).\]
        Informally, we think of $\BE_X (\fr{h}, W)$ as the category of objects ``supported on $X$''. The Verdier dual functor for $\scrA^{\oplus}_{ X} (\fr{h}, W)$ induces a contravariant autoequivalence of $\BE_X (\fr{h}, W)$ which we denote similarly.
        When $X = W$, by \cite[Remark 6.3 (1)]{ARV}, there is a canonical equivalence of categories 
        \begin{equation}\label{eq:BE_W}
                 \BE_W (\fr{h}, W) \cong \BE (\fr{h}, W).
        \end{equation}
        When $X = \{w\}$, by \cite[Lemma 4.4]{ARV}, there is a canonical equivalence of categories 
        \begin{equation}\label{eq:single_stratum}
                \BE_{\{w\}} (\fr{h}, W) \cong K^b \Free^{\fg, \Z} (R),
        \end{equation}
        where $\Free^{\fg, \Z} (R)$ denotes the additive category of graded free finitely generated graded left $R$-modules.
        We denote by $b_w \in \BE_{\{w\}} (\fr{h}, W)$ the object corresponding to $R$ in $K^b \Free^{\fg, \Z} (R)$.
        
        Given locally closed subsets $Y \subseteq W$ and $X \subseteq Y$ with $X$ finite, there are ``pushforward'' functors
        \[(i_X^Y)_! : \BE_X (\fr{h}, W) \to \BE_Y (\fr{h}, W) \qquad\text{and}\qquad (i_X^Y)_* : \BE_X (\fr{h}, W) \to \BE_Y (\fr{h}, W).\]
        When $Y$ is finite or $X$ is closed, there are also ``pullback'' functors
        \[(i_X^Y)^! : \BE_Y (\fr{h}, W) \to \BE_X (\fr{h}, W) \qquad\text{and}\qquad (i_X^Y)^* : \BE_Y (\fr{h}, W) \to \BE_X (\fr{h}, W). \]
        Here $(i_X^Y)^!$ is the right adjoint of $(i_X^Y)_!$ and $(i_X^Y)^*$ is the left adjoint of $(i_X^Y)_*$.
        These functors are constructed in \cite[\S5.4]{ARV}. They enjoy many properties reminiscent of pullback and pushforward functors of constructible sheaves along locally closed inclusions.
        We list some of these properties below.
        \begin{enumerate}
                \item When $X$ is closed, there is a recollement diagram corresponding to the pair $(X, Y \setminus X)$. See \cite[Proposition 5.6]{ARV}.
                \item The $!$-functors and $*$-functors are exchanged under Verdier duality (\!\cite[(5.15)]{ARV}), i.e., there are natural isomorphisms $\DD \circ (i_X^Y)_! = (i_X^Y)_* \circ \DD$ and $\DD \circ (i_X^Y)^! = (i_X^Y)^* \circ \DD$.
                \item They are compatible with composition (\!\cite[Lemma 5.12]{ARV}). In other words, if $X \subseteq Y \subseteq Z \subseteq W$ are locally closed subsets with $Z$ finite, then there are canonical isomorphisms $(i_X^Y)^? \circ (i_Y^Z)^? \cong (i_X^Z)^?$ and $(i_Y^Z)_? \circ (i_X^Y)_? \cong (i_X^Z)_?$ where $? \in \{*,!\}$.
                \item If $X$ is closed, then $(i_X^Y)_! = (i_X^Y)_*$. Likewise, if $X$ is open, then $(i_X^Y)^* = (i_X^Y)^!$.
                \item Assume that $W$ is finite.
                For an object $\scrF \in \BE (\fr{h}, W)$, we define its \emph{support} as the closed set 
                \[\supp (\scrF) \coloneq \overline{\{w \in W \mid (i_{\{w\}})^* \scrF \neq 0 \} } \subseteq W.\]
                If $Z = \supp (\scrF)$, then there is a natural isomorphism $\scrF \cong (i_Z)_! (i_Z)^* \scrF$.
        \end{enumerate}
        When $Y = W$, we omit the superscript $Y$ from the notation for these functors.

        We will not give a construction of the pullback and pushforward in complete generality. However, we will review the construction of $(i_X^Y)_!$ when $X$ is closed and $(i_X^Y)^*$ when $X$ is open.
        In these cases, no finiteness assumptions are needed.
        Let $Y = Y_0 \setminus Y_1$ for some closed subsets $Y_1 \subset Y_0 \subset W$.
        \begin{enumerate}
                \item Assume that $X$ is closed. 
                We can find some closed subset $X_0 \subseteq Y_0$ such that $X = X_0 \setminus (X_0 \cap Y_1)$.
                The natural embedding
                \[\scrA^{\oplus}_{ X_0} (\fr{h}, W) \subset \scrA^{\oplus}_{ Y_0} (\fr{h}, W)\]
                induces a functor
                \[ (i_X^Y)_! : \scrA^{\oplus}_{ X} (\fr{h}, W) \cong \scrA^{\oplus}_{ X_0} (\fr{h}, W) /\!\!/ \scrA^{\oplus}_{ X_0 \cap Y_1} (\fr{h}, W) \to \scrA^{\oplus}_{ Y_0} (\fr{h}, W) /\!\!/ \scrA^{\oplus}_{ Y_1} (\fr{h}, W) \cong \scrA^{\oplus}_{ Y} (\fr{h}, W).\]
                The functor $(i_X^Y)_!$ is canonically independent of the choice of $X_0$ and $Y_0$.
                This functor induces the desired $!$-pushforward
                \[(i_X^Y)_! : \BE_X (\fr{h}, W) \to \BE_Y (\fr{h}, W).\] 
                \item Assume that $X$ is open. Let $Z = Y \setminus X$ be the complementary closed subset. We can find $Z_0 \subset Y_0$ closed such that $Z = Z_0 \setminus (Z_0 \cap Y_1)$. Observe that $X = Y_0 \setminus (Z_0 \cup Y_1)$.
                There is a natural quotient functor
                 \[ (i_X^Y)^* : \scrA^{\oplus}_{ Y} (\fr{h}, W) \cong \scrA^{\oplus}_{ Y_0} (\fr{h}, W) /\!\!/ \scrA^{\oplus}_{ Y_1} (\fr{h}, W) \to \scrA^{\oplus}_{ Y_0} (\fr{h}, W) /\!\!/ \scrA^{\oplus}_{ Z_0 \cup Y_1} (\fr{h}, W) \cong \scrA^{\oplus}_{ X} (\fr{h}, W).\]
                 This functor is again independent of the choice of $X_0$ and $Z_0$. Moreover, it induces the desired $*$-pullback
                  \[(i_X^Y)^* : \BE_Y (\fr{h}, W) \to \BE_X (\fr{h}, W).\] 
        \end{enumerate}

        \subsection{Standard and Costandard Objects}

        Let $X \subseteq W$ be a locally closed subset and $w \in X$.
        Recall that there is a canonical object $b_w \in \BE_{\{w\}} (\fr{h}, W)$ corresponding to $R$ under the equivalence of categories (\ref{eq:single_stratum}).
        We define the \emph{standard} and \emph{costandard} objects in $\BE_X (\fr{h}, W)$ respectively by
        \[\Delta_w^X \coloneq (i_{\{w\}}^X)_! b_w \qquad\text{and}\qquad \nabla_w^X \coloneq (i_{\{w\}}^X)_* b_w.\]
        When $X = W$, we omit the superscript from the notation for these objects.
        
        By \cite[Lemma 6.9]{ARV}, $\BE_X (\fr{h}, W)$ is generated as a triangulated category by the objects $\Delta_w^X (n)$ with $w \in X$ and $n \in \Z$, or alternatively by the objects $\nabla_w (n)$ with $w \in X$ and $n \in \Z$.

        The $\star$-product of the standard and costandard objects is somewhat governed by the following proposition.
        \begin{proposition}[\!\!{\cite[Proposition 6.11]{ARV}}]\label{prop:arv_convolution}
                Let $x,y \in W$.
                \begin{enumerate}
                        \item If $\ell (xy) = \ell (x) + \ell (y)$, then we have isomorphisms
                        \[\Delta_{xy} \cong \Delta_x \star \Delta_y \qquad\text{and}\qquad \nabla_{xy} \cong \nabla_x \star \nabla_y.\]
                        \item We have isomorphisms 
                        \[\Delta_x \star \nabla_{x^{-1}} \cong \Delta_e \cong \nabla_{x^{-1}} \star \Delta_x.\]
                \end{enumerate}
        \end{proposition}

        \subsection{Rigidity}

        Let $A$ be a graded commutative ring. We will write
        \[(-)^{\op} : \grbim{A} \to \grbim{A}\]
        for the autoequivalence which swaps the left and right action on a graded $A$-bimodule.

        We can use the above functor to construct an autoequivalence
        \[(-)^{\op} : \scrC (\fr{h}, W) \to \scrC (\fr{h}, W).\]
        To do so, we must specify the localization data for $M^{\op}$. Note that $M_Q^w$ is in fact a graded $Q$-bimodule, and so we define $(M^{\op})_Q^w = (M_Q^{w^{-1}})^{\op}$.
        We also observe that there is an isomorphism of graded $Q$-bimodules $Q \otimes_R M \cong M \otimes_R Q$. Explicitly, $M \otimes_R Q$ is a graded $Q$-bimodule since the $M_Q^w$'s are graded $Q$-bimodules.
        As a result, there is a homomorphism $Q \otimes_R M \to M \otimes_R Q$ which can easily be checked to be an isomorphism.
        We can then define a morphism of graded $Q$-bimodules $\xi_{M^{\op}} : M^{\op} \otimes_R Q \stackrel{\sim}{\to} \bigoplus_{w \in W} (M^{\op})_Q^w$ by the composition
        \[M^{\op} \otimes_R Q \cong (Q \otimes_R M)^{\op} \cong (M \otimes_R Q)^{\op} \stackrel{\xi_{M}^{\op}}{\to} \bigoplus_{w \in W} (M_Q^w)^{\op} = \bigoplus_{w \in W} (M^{\op})_Q^w.\]
        It can be readily seen from the definitions that $(M^{\op}, ((M^{\op})_Q^w)_{w \in W}, \xi_{M^{\op}})$ is an object in $\scrC (\fr{h}, W)$. 
        Likewise, it can be checked from the definitions that this process is functorial, and hence, produces an autoequivalence of $\scrC (\fr{h}, W)$.
        This functor is anti-monoidal in the sense that there is a natural isomorphism
        \begin{equation}M^{\op} \star N^{\op} \cong (N \star M)^{\op},\end{equation}
        for $M,N \in \scrC (\fr{h}, W)$.

        It is easy to see that $B_s^{\op} = B_s$. Therefore, for any sequence $\uw = (s_1, \ldots, s_k)$ of simple reflections, anti-monoidality implies that $B_{\uw}^{\op} = B_{\overline{\uw}}$, where $\overline{\uw} = (s_k, \ldots, s_1)$ denotes the reversed expression.
        In particular, $(-)^{\op}$ restricts to an autoequivalence
        \[(-)^{\op} : \scrA^{\oplus} (\fr{h}, W) \stackrel{\sim}{\to} \scrA^{\oplus} (\fr{h}, W). \]
        
        \begin{remark}
                By passing from Abe--Bott--Samelson bimodules to Elias--Williamson diagrammatics via \cite[Theorem 5.6]{Abe19}, one observes that $(-)^{\op}$ corresponds to the horizontal reflection of morphisms.
        \end{remark}

        We will now study how $(-)^{\op}$ interacts with the mixed derived category.
        Let $X \subseteq W$ be locally closed. We can find closed subsets $X_1 \subset X_0 \subset W$ such that $X = X_0 \setminus X_1$. Note that $X^{-1} = X_0^{-1} \setminus X_1^{-1}$.
        The functor $(-)^{\op}$ restricts to a functor
        \[(-)^{\op} : \scrA^{\oplus}_{X_0} (\fr{h}, W) \to \scrA^{\oplus}_{X_0^{-1}} (\fr{h}, W)\]
        which, by passing to naive quotients, defines a functor
        \[(-)^{\op} : \scrA^{\oplus}_{X} (\fr{h}, W) \to \scrA^{\oplus}_{X^{-1}} (\fr{h}, W). \]
        By taking bounded homotopy categories, we also get an equivalence of categories
        \[(-)^{\op} : \BE_X (\fr{h}, W) \to \BE_{X^{-1}} (\fr{h}, W).\]

        \begin{lemma}\label{lem:refl_commutes_with_sheaf_functors}
                Let $X \subseteq Y \subseteq W$ be finite locally closed subsets. 
                Then for $? \in \{*,!\}$, there are natural isomorphisms $(-)^{\op} \circ (i_X^Y)_? \cong (i_{X^{-1}}^{Y^{-1}})_? \circ (-)^{\op}$ and $(-)^{\op} \circ (i_X^Y)^? \cong (i_{X^{-1}}^{Y^{-1}})^? \circ (-)^{\op}$. 
        \end{lemma}
        \begin{proof}
                Consider the closure $\overline{X}$ of $Y$. We have that $X \subseteq \overline{X} \subset Y$. For $? \in \{*,!\}$, there are canonical isomorphisms
                \[(i_X^Y)_? \cong (i_{\overline{X}}^Y)_? \circ (i_X^{\overline{X}})_? \qquad\text{and}\qquad (i_X^Y)^? \cong  (i_X^{\overline{X}})^? \circ (i_{\overline{X}}^Y)^?.\]
                Since $X$ is open in $\overline{X}$ and $\overline{X}$ is closed in $Y$, it suffices to prove the lemma when $X$ is either open or closed.

                Assume that $X$ is closed. Let $Y_0 \subset Y_1 \subset W$ be closed subsets such that $Y = Y_0 \setminus Y_1$.
                We can then find some closed subset $X_0 \subseteq Y_0$ such that $X = X_0 \setminus (X_0 \cap Y_1)$. 
                In this case, $(i_X^Y)_! =(i_X^Y)_*$ is induced from the fully faithful functor
                \[ \scrA^{\oplus}_{X} (\fr{h}, W) \cong \scrA^{\oplus}_{X_0} (\fr{h}, W) /\!\!/ \scrA^{\oplus}_{ X_0 \cap Y_1} (\fr{h}, W) \to \scrA^{\oplus}_{ Y_0} (\fr{h}, W) /\!\!/ \scrA^{\oplus}_{ Y_1} (\fr{h}, W) \cong \scrA^{\oplus}_{ Y} (\fr{h}, W).\]
                It can then be checked from the definitions that $(-)^{\op} \circ (i_X^Y)_! \cong (i_{X^{-1}}^{Y^{-1}})_! \circ (-)^{\op}$. Since $(-)^{\op}$ is an involution, we can take adjoints to conclude that $(-)^{\op} \circ (i_X^Y)^? \cong (i_{X^{-1}}^{Y^{-1}})^? \circ (-)^{\op}$ for $? \in \{*,!\}$ as well.
        
                Now assume that $X$ is open. Let $Z = Y \setminus X$ be the complementary closed subset. 
                Take $Y = Y_0 \setminus Y_1$ as above.  
                We can find $Z_0 \subset Y_0$ closed such that $Z = Z_0 \setminus (Z_0 \cap Y_1)$. Note that $X = Y_0 \setminus (Z_0 \cup Y_1)$. 
                In this case, $(i_X^Y)^! =(i_X^Y)^*$ is induced from the full functor
                \[ \scrA^{\oplus}_{Y} (\fr{h}, W) \cong \scrA^{\oplus}_{Y_0} (\fr{h}, W) /\!\!/ \scrA^{\oplus}_{ Y_1} (\fr{h}, W) \to \scrA^{\oplus}_{ Y_0} (\fr{h}, W) /\!\!/ \scrA^{\oplus}_{ Z_0 \cup Y_1} (\fr{h}, W) \cong \scrA^{\oplus}_{ X} (\fr{h}, W).\]
                Again it can be checked from the definitions that  $(-)^{\op} \circ (i_X^Y)^* \cong (i_{X^{-1}}^{Y^{-1}})^* \circ (-)^{\op}$. Since $(-)^{\op}$ is an involution, we can take adjoints to conclude that $(-)^{\op} \circ (i_X^Y)^? \cong (i_{X^{-1}}^{Y^{-1}})^? \circ (-)^{\op}$ for $? \in \{*,!\}$ as well.
        \end{proof}

        Recall that $B_s = R \otimes_{R^s} R (1)$ and $B_s \otimes_R B_s \cong R \otimes_{R^s} R \otimes_{R^s} R (1)$.
        Define $\partial_s : R \to R^s$ by $\partial_s (f) = (f-s(f))/\alpha_s$ for $f \in R$. By Demazure surjectivity for $\fr{h}$, we can fix some $\delta_s \in R$ such that $\partial_s (\delta_s) = 1$.
        For $s \in S$, there are morphisms
        \begin{align*}
                \cup_s &: R \to B_s \otimes_R B_s, &  1 &\mapsto \delta_s \otimes 1 \otimes 1 - 1 \otimes 1 \otimes s(\delta_s), \\
                \cap_s &: B_s \otimes_R B_s \to R, &  f \otimes g \otimes h &\mapsto f \partial_s (g) h.
        \end{align*}
        More generally, for a sequence $\uw = (s_1, \ldots, s_k)$ of simple reflections, we may define morphisms
        \[\cup_{\uw} \coloneq (\id_{B_{(s_1, \ldots, s_{k-1})}} \otimes \cup_{s_k} \otimes \id_{B_{(s_{k-1}, \ldots ,s_1)}}) \circ \ldots \circ (\id_{B_{s_1}} \otimes \cup_{s_2} \otimes \id_{B_{s_1}}) \circ \cup_{s_1} : R \to B_{\uw} \otimes_R B_{\overline{\uw}},\]
        \[\cap_{\uw} \coloneq \cap_{s_k} \circ (\id_{B_{s_k}} \otimes \cap_{s_{k-1}} \otimes \id_{B_{s_k}}) \circ \ldots \circ (\id_{B_{(s_k, \ldots, s_{2})}} \otimes \cap_{s_1} \otimes \id_{B_{(s_{2}, \ldots, s_{k})}}) : B_{\overline{\uw}} \otimes_R B_{\uw} \to R. \]

        We can then define morphisms $\cup_{\scrF} : R \to \scrF \star \DD (\scrF)^{\op}$ and $\cap_{\scrF} : \DD(\scrF)^{\op} \star \scrF \to R$ for any $\scrF \in \BE (\fr{h}, W)$ by extending from the Bott--Samelson objects.
        
        \begin{lemma}\label{lem:duals}
                Let $\scrF \in \BE (\fr{h}, W)$. Then there are equalities
                \[(\id_{\scrF} \star \cap_{\scrF}) \circ (\cup_{\scrF} \star \id_{\scrF}) = \id_{\scrF} = (\cap_{\DD(\scrF)^{\op}} \star \id_{\scrF}) \circ (\id_{\scrF} \star \cup_{\DD(\scrF)^{\op} }).\]
        \end{lemma}
        \begin{proof}
                It suffices to prove the equalities when $\scrF = B_{\uw}$ for a sequence $\uw$ of simple reflections. By monoidality, we can further reduce to the case when $\scrF = B_s$.
                This case of the lemma can then be checked from the definitions.
        \end{proof}

        \begin{remark}
                Lemma \ref{lem:duals} states that $\BE (\fr{h}, W)$ is a rigid (in fact, pivotal) monoidal category. The operation of taking duals corresponds to the functor
                \[\BE (\fr{h}, W) \to \BE (\fr{h}, W), \qquad \scrF \mapsto \DD (\scrF)^{\op}.\]
                In particular, there is a natural isomorphism
                \[\Hom (\scrF, \scrG \star \scrH) \cong \Hom (\DD (\scrG)^{\op} \star \scrF, \scrH)\]
                for all $\scrF, \scrG, \scrH \in \BE (\fr{h}, W)$.
        \end{remark}

        \subsection{Parabolic Induction} 

         Fix $I \subseteq S$ and write $W_I$ for the parabolic subgroup of $W$ generated by $I$.
        We can now consider the composition,
        \[\begin{tikzcd}
{\iota : \BE (\fr{h}, W_I) \cong \BE_{W_I} (\fr{h}, W)} \arrow[rr, "(i_{W_I})_*"] & & {\BE_{W} (\fr{h}, W) \cong \BE (\fr{h}, W).}
\end{tikzcd}\]
        We call $\iota$ the \emph{parabolic induction functor}. 
        Note that $\iota$ is fully faithful and monoidal. When $W_I$ is finite, parabolic induction admits both a left and right adjoint, denoted $\iota^L$ and $\iota^R$ respectively. 
        They are given by the compositions
        \[\begin{tikzcd}
{\iota^L : \BE (\fr{h}, W) \cong \BE_{W} (\fr{h}, W)} \arrow[rr, "(i_{W_I})^*"] &  & {\BE_{W_I} (\fr{h}, W) \cong \BE (\fr{h}, W_I),}
\end{tikzcd}\]
        \[\begin{tikzcd}
{\iota^R : \BE (\fr{h}, W) \cong \BE_{W} (\fr{h}, W)} \arrow[rr, "(i_{W_I})^!"] &  & {\BE_{W_I} (\fr{h}, W) \cong \BE (\fr{h}, W_I).}
\end{tikzcd}\]
We call the functors $\iota^L$ and $\iota^R$ the \emph{parabolic restriction functors}.

        We will now assume that $W$ itself is finite.
        Let $X$ be a left $W_I$-stable subset of $W$. If $X$ is closed, then by \cite[Corollary 6.4]{ARV}, $\iota (\scrF) \star (i_X)_* (\scrG)$ is in the essential image of $(i_X)_*$ for $\scrF \in \BE (\fr{h}, W_I)$ and $\scrG \in \BE_X (\fr{h}, W)$.
        We can then define a bifunctor
        \begin{equation}\label{eq:left_WI_action}
                (-) \wistar (-) : \BE (\fr{h}, W_I) \times \BE_X (\fr{h}, W) \to \BE_X (\fr{h}, W) \qquad (\scrF, \scrG) \mapsto (i_X)^* (\iota (\scrF) \star (i_X)_* (\scrG)).
        \end{equation}
        The discussion above implies that $\wistar$ makes $\BE_X (\fr{h}, W)$ into a left $\BE (\fr{h}, W_I)$-module.
        For a general $X$, we can write $X = X_0 \setminus X_1$ with $X_1 \subset X_0 \subset W$ closed and left $W_I$-stable.
        By \cite[Remark 5.7]{ARV}, the functor $(i_X^{X_0})^*$ identifies $\BE_X (\fr{h}, W)$ with the Verdier quotient $\BE_{X_0} (\fr{h}, W) / \BE_{X_1} (\fr{h}, W)$.
        In particular, one can readily check that $\iota (\mathcal{F}) \star (i_X)_* \mathcal{G}$ is in the essential image of $(i_X)_*$.
        It also follows from the Verdier quotient description that $\wistar$ induces a left action of $\BE (\fr{h}, W_I)$ on $\BE_X (\fr{h}, W)$, which will also be denoted by $\wistar$. 
        Since $(i_X^{X_0})_*$ is a section of $(i_X^{X_0})^*$, one has that the induced action is defined by the same formula as (\ref{eq:left_WI_action}). 
        Alternatively, since $(i_X^{X_0})_!$ is also a section, it can be defined by the variation of (\ref{eq:left_WI_action}) where the $(i_X)_*$ is replaced by $(i_X)_!$.
        The rigid monoidal structure on $\BE (\fr{h}, W_I)$ from Lemma \ref{lem:duals} implies that there are natural isomorphisms
        \begin{equation}\label{eq:rigid_hom_module}
                \Hom (\scrF \star^I \scrG, \scrH) \cong \Hom (\scrG, \DD (\scrF)^{\op} \star^I \scrH)
        \end{equation}
        where $\scrF \in \BE (\fr{h}, W_I)$ and $\scrG, \scrH \in \BE_X (\fr{h}, W)$.
        
        Likewise, if $X$ is instead right $W_I$-stable, an analogous argument allows us to construct a bifunctor 
        \[(-) \wistar (-) : \BE (\fr{h}, W) \times \BE_X (\fr{h}, W_I) \to \BE_X (\fr{h}, W) \qquad (\scrG, \scrF) \mapsto (i_X)^* ((i_X)_* (\scrG) \star \iota (\scrF))\]
        which makes $\BE_X (\fr{h}, W)$ into a right $\BE (\fr{h}, W_I)$-module.
        
        \begin{proposition}\label{prop:pull_push_module_functors}
                Assume that $W$ is finite, and let $X \subseteq Y \subseteq W$ be locally closed subsets. Let $\scrF \in \BE (\fr{h}, W_I)$, $\scrG \in \BE_X (\fr{h}, W)$,  and $\scrH \in \BE_Y (\fr{h}, W)$.
                \begin{enumerate}
                        \item Assume that $X$ and $Y$ are left $W_I$-stable. Then there are natural isomorphisms
                        \begin{align*}
                                \scrF \wistar (i_X^Y)_* \scrG &\cong (i_X^Y)_* (\scrF \wistar \scrG), & \scrF \wistar (i_X^Y)_! \scrG &\cong (i_X^Y)_! (\scrF \wistar \scrG), \\
                                \scrF \wistar (i_X^Y)^* \scrH &\cong (i_X^Y)^* (\scrF \wistar \scrH), & \scrF \wistar (i_X^Y)^! \scrH &\cong (i_X^Y)^! (\scrF \wistar \scrH).
                        \end{align*}
                        \item Assume that $X$ and $Y$ are right $W_I$-stable. Then there are natural isomorphisms
                        \begin{align*}
                                (i_X^Y)_* \scrG \wistar \scrF &\cong (i_X^Y)_* (\scrG \wistar \scrF), & (i_X^Y)_! \scrG \wistar \scrF &\cong (i_X^Y)_! (\scrF \wistar \scrG), \\
                                (i_X^Y)^* \scrH \wistar \scrF &\cong (i_X^Y)^* (\scrH \wistar \scrF), & (i_X^Y)^! \scrH \wistar \scrF &\cong (i_X^Y)^! (\scrH \wistar \scrF).
                        \end{align*}
                \end{enumerate} 
        \end{proposition}
        \begin{proof}
                We will just prove (1) as (2) follows from a symmetric argument.
                First, we will show the isomorphism
                \[ \scrF \wistar (i_X^Y)_* \scrG \cong (i_X^Y)_* (\scrF \wistar \scrG).\]
                The other pushforward isomorphism follows from a similar argument using the earlier observation that $(i_X)_*$ can be replaced by $(i_X)_!$ in the definition of the $\wistar$-bifunctor from (\ref{eq:left_WI_action}).
                We can check the isomorphism from the definitions and using the observation that $\iota (\mathcal{F}) \star (i_X)_* \mathcal{G}$ is in the essential image of $(i_X)_*$.
                \begin{align*}
                        \mathcal{F} \wistar (i_X^Y)_* \scrG &\cong (i_Y)^* (\iota (\scrF) \star (i_X)_* \scrG) \\
                        &\cong (i_X)_* (i_X)^* (i_Y)^*  (\iota (\scrF) \star (i_X)_* \scrG) \\
                        &\cong (i_X^Y)_* (\scrF \wistar \scrG).
                \end{align*}
                
                Next, we will show there is an isomorphism
                \[ \scrF \wistar (i_X^Y)^* (\scrH) \cong (i_X^Y)^* (\scrF \wistar \scrH).\]
                The remaining isomorphism follows from a similar argument.
                Let $\scrE \in \BE_X (\fr{h}, W)$. We can then compute
                \begin{align*}
                        \Hom (\scrF \wistar (i_X^Y)^* \scrH, \scrE) &\cong \Hom ((i_X^Y)^* \scrH, \DD (\scrF)^{\op} \wistar \scrE) \\
                        &\cong \Hom ( \scrH, (i_X^Y)_* ( \DD (\scrF)^{\op}  \wistar \scrE )) \\
                        &\cong \Hom (\scrH,  \DD (\scrF)^{\op} \wistar (i_X^Y)_* \scrE) \\
                        &\cong \Hom (\scrF \wistar \scrH,  (i_X^Y)_* \scrE) \\
                        &\cong \Hom ((i_X^Y)^* (\scrF \wistar \scrH), \scrE).
                \end{align*}
                Here the first and fourth isomorphisms follow from (\ref{eq:rigid_hom_module}) and the third isomorphism follows from the earlier observation that $(i_X^Y)_*$ is left $\BE (\fr{h}, W_I)$-linear.
        \end{proof}

        For this paper, we will only need the special case of the proposition where $X = W_I$ and $Y = W$. This case is provided by the following corollary which simply unpacks the notation.

        \begin{corollary}\label{cor:par_linearity}
               There are natural isomorphisms
                   \begin{align*}
                \scrF \star \iota^L (\scrG) &\cong \iota^L (\iota (\scrF) \star \scrG), & \scrF \star \iota^R (\scrG) &\cong \iota^R (\iota (\scrF) \star \scrG), \\
                \iota^L (\scrG) \star \scrF &\cong \iota^L (\scrG \star \iota (\scrF)), & \iota^R (\scrG) \star \scrF &\cong \iota^R (\scrG \star \iota (\scrF)),
        \end{align*}
        for all $\scrF \in \BE (\fr{h}, W_I)$ and $\scrG \in \BE (\fr{h}, W)$.
        \end{corollary}

        \section{Relative Serre Duality}

        For the remainder of the paper, we will assume that $W$ is finite.
        Fix $I \subseteq S$, and write $W_I$ for the parabolic subgroup of $W$ generated by $I$. 
        We denote by $w_0$ the longest element in $W$ and $w_I$ the longest element in $W_I$.
        Define the \emph{full twist} objects,
        \[ \FT_W \coloneq \Delta_{w_0} \star \Delta_{w_0} \qquad\text{and}\qquad \FT_W^{-1} \coloneq \nabla_{w_0} \star \nabla_{w_0}.\]
        By Proposition \ref{prop:arv_convolution}, there are isomorphisms $\FT_W \star \FT_W^{-1} \cong \Delta_e \cong \FT_W^{-1} \star \FT_W$.
        We also define the \emph{relative full twist} object 
        \[ \FT_{W, I} \coloneq \iota (\FT_{W_I}^{-1}) \star \FT_W. \]

        \begin{theorem}[Relative Serre Duality]\label{thm:rel_serre_duality}
                There are natural isomorphisms of functors 
                \[\iota^L (\FT_{W, I} \star -) \cong  \iota^R \cong \iota^L (- \star \FT_{W, I}).\]
        \end{theorem}

        \begin{corollary}\label{cor:rel_full_twist_centrality}
                The relative full twist is canonically central with respect to objects in $\BE (\fr{h}, W_I)$. 
                I.e., there is a natural isomorphism
                \[\FT_{W, I} \star \iota (\scrF) \cong \iota (\scrF) \star \FT_{W, I}\]
                for $\scrF \in \BE (\fr{h}, W_I)$.
        \end{corollary}

        The above theorem and its corollary is a restatement of Conjecture \ref{conj:rel_serre_duality_intro} which has been broken up to improve the exposition.
        The proof of these statements will require some preparation.
        Our argument closely follows Ho and Li's argument \cite{HLnew}, or more precisely, an earlier unpublished pre-print \cite{HLold}.

        \begin{lemma}\label{lem:ft_on_kernels}
                The functors 
                \[ \FT_W \star (-): \BE (\fr{h}, W) \to \BE (\fr{h}, W) \qquad\text{and}\qquad (-) \star \FT_W: \BE (\fr{h}, W) \to \BE (\fr{h}, W)\] 
                restrict to equivalences of categories $\ker (\iota \iota^R) \to \ker (\iota \iota^L)$.
        \end{lemma}
        \begin{proof}
                Our proof for the lemma is derived from \cite[Lemma 3.3.2]{HLnew}.
                The argument for $(-) \star \FT_W$ is symmetric to the argument for $\FT_W \star (-)$ and is omitted.
                Since $\FT_W \star (-)$ is an autoequivalence of $\BE (\fr{h}, W)$ with inverse $\FT_W^{-1} \star (-)$, the content of the lemma is that $\FT_W \star \ker (\iota \iota^R) = \ker (\iota \iota^L)$.
                
                By recollement (cf., \cite[Lemma 6.9]{ARV}), 
                \[\ker (\iota \iota^R) = \langle \nabla_x : x \notin W_I\rangle_{(1), \Delta}\]
                where the right-hand side denotes the triangulated subcategory generated by $\nabla_x (n)$ for $x \notin W_I$ and $n \in \Z$.
                Likewise, 
                \[\ker (\iota \iota^L) = \langle \Delta_x : x \notin W_I\rangle_{(1), \Delta}\]
                By Proposition \ref{prop:arv_convolution}, there is an isomorphism $\Delta_{w_0} \star \nabla_x \cong \Delta_{w_0 x}$ for any $x \in W$.
                As a result,
                \begin{equation}\label{eq:ft_on_kernels}
                        \Delta_{w_0} \star \ker (\iota \iota^R) = \langle \Delta_{w_0 x} : x \notin W_I\rangle_{(1), \Delta}.
                \end{equation}
                Let $Z = \{ w_0 x : x \notin W_I \} \subseteq W$. Note that $Z$ is closed in $W$. Indeed, if $z \leq y$ with $y \in Z$, then $w_0 y \leq w_0 z$.
                Since $W_I$ is closed and $w_0 y \notin W_I$, we must have that $w_0 z \notin W_I$ as well. In other words, $z \in Z$.
                As a result, from (\ref{eq:ft_on_kernels}) and recollement, we see that
                \[ \Delta_{w_0} \star \ker (\iota \iota^R) = (i_Z)_* \BE_{Z} (\fr{h}, W) = \langle \nabla_y : y \in Z \rangle_{(1), \Delta}.\]
                We then have that
                \[\FT_{W} \star \ker (\iota \iota^R) = \langle \Delta_{w_0} \star \nabla_y : y \in Z \rangle_{(1), \Delta} = \langle \Delta_x : x \in W_I\rangle_{(1), \Delta} = \ker (\iota \iota^L)\]
                as desired.
        \end{proof}

        \begin{lemma}\label{lem:rel_ft_on_kernels}
                The functors 
                \[\FT_{W, I} \star (-) : \BE (\fr{h}, W) \to \BE (\fr{h}, W) \qquad\text{and}\qquad (-) \star \FT_{W, I} : \BE (\fr{h}, W) \to \BE (\fr{h}, W)\]
                restrict to equivalences of categories $\ker (\iota \iota^R) \to \ker (\iota \iota^L)$.
        \end{lemma}
        \begin{proof}
                Our proof for the lemma is derived from \cite[Proposition 3.1.1]{HLold}.
                 The argument for $(-) \star \FT_{W, I}$ is almost identical to the argument for $\FT_{W, I} \star (-)$ and is omitted.
                By Lemma \ref{lem:ft_on_kernels}, it suffices to prove that $\iota (\FT_{W_I}^{-1}) \star \ker (\iota \iota^L) = \ker (\iota \iota^L)$.
                For all $\scrF \in \BE (\fr{h}, W_I)$ and $\scrG \in \BE (\fr{h}, W)$, we have that 
                \begin{equation}\label{eq:rel_ft_on_kernels}
                        \iota (\scrF) \star \iota \iota^L (\scrG) \cong \iota (\scrF \star \iota^L \scrG) \cong \iota \iota^L (\iota (\scrF) \star \scrG).
                \end{equation}
                Here the first isomorphism follows from $\iota$ being monoidal and the second isomorphism follows from Corollary \ref{cor:par_linearity}.
                When $\scrF = \FT_{W_I}^{-1}$, equation (\ref{eq:rel_ft_on_kernels}) implies that $\iota (\FT_{W_I}^{-1}) \star \ker (\iota \iota^L) \subseteq \ker (\iota \iota^L)$.
                Likewise for $\scrF = \FT_{W_I}$, equation (\ref{eq:rel_ft_on_kernels}) implies that $\iota (\FT_{W_I}^{-1}) \star \ker (\iota \iota^L) \supseteq \ker (\iota \iota^L)$.
        \end{proof}

        \begin{lemma}\label{lem:support_tech_lemma}
                Let $\scrF \in \BE (\fr{h}, W)$ such that $\supp (\scrF) \subseteq Z \coloneq \{x \in W \mid x < w_I w_0\}$. 
                Then $\iota^L (\Delta_{w_I} \star \scrF \star \Delta_{w_0} ) = 0$. 
        \end{lemma}
        \begin{proof}
                By \cite[Lemma 6.9]{ARV}, the category $\BE_Z (\fr{h}, W)$ is generated by $\nabla_x^Z (n)$ for $x < w_I w_0$ and $n \in \Z$.
                Moreover, since $Z$ is closed, we have that $(i_Z)_* \nabla_x^Z \cong \nabla_x$ for all $x \in Z$.
                By Proposition \ref{prop:arv_convolution}, there is an isomorphism $\Delta_{w_I} \star  \nabla_x \star \Delta_{w_0} \cong \Delta_{w_I} \star \Delta_{x w_0}$.
                We see that 
                \[\Delta_{w_I} \star (i_Z)_* \BE_Z (\fr{h}, W) \star \Delta_{w_0} = \langle \Delta_{w_I} \star \Delta_{y} : y > w_I \rangle_{(1), \Delta}.\]
                Let $y \in W$ with $y > w_I$. There is a unique decomposition $y = v_y u_y$ where $v_y \in W_I$ and $u_y \in W$ a minimal length representative for one of the cosets in $W_I \backslash W$.
                The condition that $y > w_I$ implies that $u_y \neq e$. Moreover, $\ell (y) = \ell (v_y) + \ell (u_y)$.
                Since $\Delta_{w_I} \star \Delta_{v_y}$ is supported on $W_I$, we have that
                \[\langle \Delta_{w_I} \star \Delta_{y} : y > w_I \rangle_{(1), \Delta} = \langle \Delta_{w_I} \star \Delta_{v_y} \star \Delta_{u_y} : y > w_I \rangle_{(1), \Delta} \subseteq \langle \Delta_x \star \Delta_{u_y} : x \in W_I, y > w_I\rangle_{(1), \Delta}.\]
                Our condition on $u_y$ ensures that $\ell (xu_y) = \ell (x) + \ell (u_y)$ for all $x \in W_I$. 
                We then again have by Proposition \ref{prop:arv_convolution} that $\Delta_x \star \Delta_{u_y} \cong \Delta_{xu_y}$ for all $x \in W_I$.
                Additionally, since $u_y \neq e$, $xu_y \notin W_I$ for any $x \in W_I$. In other words, $\iota^L (\Delta_{xu_y}) = 0$.
                Since this holds for all $x \in W_I$, we can conclude that $\iota^L (\Delta_{w_I} \star \scrF \star \Delta_{w_0} ) = 0$ for all $\scrF \in \BE (\fr{h}, W)$ with $\supp (\scrF) \subseteq Z$. 
        \end{proof}

        The following lemma appeared as \cite[Proposition 3.1.2]{HLold}, however its proof is subtly flawed. 
        This was corrected in \cite{HLnew} where a version of the proposition appears with $I = \emptyset$ (see \cite[Proposition 3.1.2]{HLnew}).
        We give an alternate argument for the original proposition with a general $I \subseteq S$. 
        This allows us to avoid reducing to the $I = \emptyset$ case in proving Theorem \ref{thm:rel_serre_duality}.
        The reduction technique in \cite{HLnew} does not adapt well to our setting since it makes crucial use of the adjoint functor theorem for presentable $\infty$-categories.

        \begin{lemma}\label{lem:restr_of_rel_counit}
                There exists a morphism $\alpha : \FT_W \to \iota (\FT_{W_I})$ such that $\iota^L (\alpha)$ is an isomorphism.
        \end{lemma}
        \begin{proof}
                We define $\alpha$ by the composition
                \[\alpha : \FT_W = \Delta_{w_0} \star \Delta_{w_0} \cong \Delta_{w_I} \star \Delta_{w_I w_0} \star \Delta_{w_0} \stackrel{\beta}{\to} \Delta_{w_I} \star \nabla_{w_I w_0} \star \Delta_{w_0} \cong \Delta_{w_I} \star \Delta_{w_I} = \iota (\FT_{W_I})\]
                where the map $\beta$ is induced from the canonical morphism $\Delta_{w_I w_0} \to \nabla_{w_I w_0}$.
                Note that $\iota^L (\alpha)$ is an isomorphism if and only if $\iota^L (\Delta_{w_I} \star C \star \Delta_{w_0}) = 0$ where $C$ is the cone of $\beta$.
                Observe that the support of $C$ is contained in $\{x \in W \mid x < w_I w_0\}$. We are then done by Lemma \ref{lem:support_tech_lemma}.
        \end{proof}

        \begin{midsecproof}{Theorem \ref{thm:rel_serre_duality}}
                We just prove the isomorphism $\iota^L (\FT_{W, I} \star -) \cong  \iota^R$. The other isomorphism is similar.
                Let $\scrF \in \BE (\fr{h}, W)$. 
                By recollement, there is a distinguished triangle
                \[(i_{W_I})_* (i_{W_I})^! \scrF \to \scrF \to (i_{ W \setminus W_I})_* (i_{W \setminus W_I})^* \scrF \to \]
                which is equivalent to 
                \[\iota \iota^R \scrF \to \scrF \to (i_{ W \setminus W_I})_* (i_{W \setminus W_I})^* \scrF \to.\]
                We can apply $\FT_{W, I} \star (-)$ to the above triangle to obtain
                \begin{equation}\label{eq:rel_serre_duality_1}
                        \FT_{W, I} \star \iota \iota^R \scrF \to \FT_{W, I} \star \scrF \to \FT_{W, I} \star (i_{ W \setminus W_I})_* (i_{W \setminus W_I})^* \scrF \to.
                \end{equation}
                Observe that $(i_{ W \setminus W_I})_* (i_{W \setminus W_I})^* \scrF \in \ker \iota^R$, and hence by Lemma \ref{lem:rel_ft_on_kernels}, $\FT_{W, I} \star (i_{ W \setminus W_I})_* (i_{W \setminus W_I})^* \scrF  \in \ker \iota^L$.
                As a result, applying $\iota^L$ to (\ref{eq:rel_serre_duality_1}) yields an isomorphism
                \begin{equation}\label{eq:rel_serre_duality_2}
                        \iota^L (\FT_{W, I} \star \iota \iota^R \scrF) \cong \iota^L (\FT_{W, I} \star \scrF).
                \end{equation}
                On the other hand, we can use Corollary \ref{cor:par_linearity} to produce isomorphisms
                \begin{equation}\label{eq:rel_serre_duality_3}
                        \iota^L (\FT_{W, I} \star \iota \iota^R \scrF) \cong \iota^L (\FT_{W, I}) \star \iota^R (\scrF) \cong \iota^L (\iota (\FT_{W_I}^{-1}) \star \FT_W)  \star \iota^R (\scrF) \cong \FT_{W_I}^{-1} \star \iota^L (\FT_W) \star \iota^R (\scrF).
                \end{equation}
                By Lemma \ref{lem:restr_of_rel_counit}, there is also an isomorphism
                \begin{equation}\label{eq:rel_serre_duality_4}
                 \FT_{W_I}^{-1} \star \iota^L (\FT_W) \star \iota^R (\scrF) \stackrel{\sim}{\to} \FT_{W_I}^{-1} \star \FT_{W_I} \star \iota^R (\scrF) \cong \iota^R (\scrF)
                \end{equation}
                induced by $\alpha : \FT_W \to \iota (\FT_{W_I})$.
                By combining the isomorphisms in (\ref{eq:rel_serre_duality_2}), (\ref{eq:rel_serre_duality_3}), and (\ref{eq:rel_serre_duality_4}), we deduce the desired isomorphism
                \[\iota^L (\FT_{W, I} \star \scrF) \stackrel{\sim}{\to} \iota^R (\scrF).\]
                It is easy to check that all the morphisms used in the construction of this isomorphism are natural in $\scrF$. 
        \end{midsecproof}

        \begin{lemma}\label{lem:ft_commutes}
                Let $x \in W$. Then there is an isomorphism $\Delta_x \star \FT_W \cong \FT_W \star \Delta_x$.
        \end{lemma}
        \begin{proof}
                By Proposition \ref{prop:arv_convolution}, it suffices to take $x = s \in S$. We can then conjugate $\Delta_s$ by $\Delta_{w_0}$ to get an isomorphism
                \[\Delta_{w_0} \star \Delta_s \star \nabla_{w_0} \cong \Delta_{w_0} \star \nabla_{sw_0} \cong \Delta_{w_0 s w_0} \star \Delta_{w_0 s} \star \nabla_{s w_0} \cong \Delta_{w_0 s w_0}. \]
                Likewise, we can conjugate $\Delta_{w_0 s w_0}$ by $\Delta_{w_0}$ to get an isomorphism
                \[\Delta_{w_0} \star \Delta_{w_0 s w_0} \star \nabla_{w_0} \cong \Delta_{w_0} \star \Delta_{w_0 s w_0} \star \nabla_{w_0 s w_0} \star \nabla_{w_0 s} \cong \Delta_{w_0} \star \nabla_{w_0 s} \cong \Delta_s.\]
                Combining these isomorphisms, we obtain $\FT_W \star \Delta_s \star \FT_W^{-1} \cong \Delta_s$ as desired. 
        \end{proof}

        \begin{remark}
                Lemma \ref{lem:ft_commutes} can also be deduced from Rouquier's action of the braid group on $\BE (\fr{h}, W)$ constructed in \cite{Rou}.
                The lemma then follows from the observation that the full twist in the braid group is a central element.
        \end{remark}

        \begin{midsecproof}{Corollary \ref{cor:rel_full_twist_centrality}}
                 Consider the full triangulated subcategory $\scrC = \langle \FT_{W, I} \star \Delta_x : x \in W_I \rangle_{(1), \Delta}$ of $\BE (\fr{h}, W)$. 
                By Lemma \ref{lem:ft_commutes}, we also have that $\scrC = \langle \Delta_x \star \FT_{W, I} : x \in W_I \rangle_{(1), \Delta}$. In particular, the functors $\FT_{W,I} \star \iota (-)$ and $\iota (-) \star \FT_{W,I}$ factor through essentially surjective functors
                \[\begin{tikzcd}
{\BE (\fr{h}, W_I)} \arrow[rr, "{\FT_{W,I} \star \iota (-)}", shift left] \arrow[rr, "{\iota (-) \star \FT_{W,I} }"', shift right] &  & \scrC.
\end{tikzcd}\]
Since $\FT_{W,I}$ is an invertible object in $\BE (\fr{h}, W)$, both of these functors are in fact equivalences of categories. 
By Theorem \ref{thm:rel_serre_duality}, the inverse to both functors is given by $\iota^L : \scrC \to \BE (\fr{h}, W_I)$.
Therefore, we can conclude that  $\FT_{W, I} \star \iota (-) \cong \iota (-) \star \FT_{W, I}$.
        \end{midsecproof}


\begin{thebibliography}{S}

        \bibitem[Abe1]{Abe19}
        N. Abe, \emph{A bimodule description of the Hecke category}, Compos. Math. {\bf 157} (2021), no.~10, 2133--2159.

        \bibitem[Abe2]{Abe21}
        N. Abe, \emph{A homomorphism between Bott--Samelson bimodules}, Nagoya Math. J. {\bf 256} (2024), 761--784.

        \bibitem[AMRW]{AMRW}
        P.~N. Achar, S. Makisumi, S. Riche, and G. Williamson, {\it Koszul duality for Kac-Moody groups and characters of tilting modules}, J. Amer. Math. Soc. {\bf 32} (2019), no.~1, 261--310.

        \bibitem[ARV]{ARV}
        P.~N. Achar, S. Riche and C. Vay, \emph{Mixed perverse sheaves on flag varieties for Coxeter groups}, Canad. J. Math. {\bf 72} (2020), no.~1, 1--55.

        \bibitem[EH]{EH}
        B. Elias and M. Hogancamp, \emph{Drinfeld centralizers and Rouquier complexes}. Preprint \href{https://arxiv.org/abs/2412.20633}{arXiv:2412.20633}.

        \bibitem[EW]{EW}
        B. Elias and G. Williamson, \emph{Soergel calculus}. Representation Theory of the American Mathematical Society, 20(12):295–374, 2016.

        \bibitem[GHMN]{GHMN}
        E. Gorsky, M. Hogancamp, A. Mellit, and K. Nakagane, \emph{Serre duality for Khovanov--Rozansky homology}, Selecta Math. (N.S.) {\bf 25} (2019), no.~5, Paper No. 79, 33 pp.

        \bibitem[HL1]{HLgeom}
        Q. P. Ho and P. Li. \emph{Revisiting Mixed Geometry}, J. Eur. Math. Soc. (JEMS) (2025).

        \bibitem[HL2]{HLold}
        Q.~P. Ho and P. Li, \emph{Relative Serre duality for Hecke categories}. Preprint \href{https://arxiv.org/abs/2504.12798v1}{arXiv:2504.12798v1}.

        \bibitem[HL3]{HLnew}
        Q.~P. Ho and P. Li, \emph{Relative Serre duality for Hecke categories} (corrected version). Preprint \href{https://arxiv.org/abs/2504.12798v2}{arXiv:2504.12798v2}, to appear in Selecta Math (N. S.).

        \bibitem[Kho]{Kho}
        M.~G. Khovanov, \emph{Triply-graded link homology and Hochschild homology of Soergel bimodules}, Internat. J. Math. {\bf 18} (2007), no.~8, 869--885.

        \bibitem[KR]{KR}
        M.~G. Khovanov and L. Rozansky, \emph{Matrix factorizations and link homology. II}, Geom. Topol. {\bf 12} (2008), no.~3, 1387--1425.

        \bibitem[Li]{Li}
        C. Li, \emph{Serre duality and the Whitehead link}. Preprint \href{https://arxiv.org/abs/2509.22133}{arXiv:2509.22133}.

        \bibitem[Rou]{Rou}
        R. Rouquier. \emph{Derived equivalences and finite dimensional algebras}. Proceedings of the International Congress of Mathematicians, 2:191–221, 2006.

        \bibitem[San]{Sandvik}
        C. Sandvik. \emph{Soergel calculus for monodromic Hecke categories}. In preparation.

    \end{thebibliography}
\end{document}